\documentclass[a4paper,12pt]{article}
\usepackage{amsmath,amsthm,amsfonts,amssymb,graphicx}
\usepackage{xspace}
\usepackage[ruled]{algorithm}
\usepackage{algorithmic}
\numberwithin{algorithm}{section}
\newif\ifpreprint
\preprinttrue
\newcommand{\R}{\mathbb{R}}
\newcommand{\B}{\mathcal{B}}
\newcommand{\Mass}{\boldsymbol{M}}
\newcommand{\Stiff}{\boldsymbol{S}}
\newcommand{\loadf}{\boldsymbol{f}}
\newcommand{\bilinA}{\textsf{A}}
\newcommand{\iprod}[1]{\langle#1\rangle}

\newcommand{\Hdot}{\dot H}
\newcommand{\nrml}{\textnormal{\textsf{n}}}
\newcommand{\Cluster}{\textnormal{\textsf{C}}}
\newcommand{\Mluster}{\textnormal{\textsf{M}}}
\newcommand{\Tree}{\mathcal{T}}
\newcommand{\Leaf}{\textnormal{\textsf{L}}}
\newcommand{\Cover}{\mathcal{C}}
\newcommand{\MinC}{\mathcal{M}}
\newcommand{\History}{\operatorname{History}}
\newcommand{\Near}{\operatorname{Near}}
\newcommand{\Far}{\operatorname{Far}}
\newcommand{\Dist}{\operatorname{Dist}}
\newcommand{\Bbar}{\bar\B}
\newcommand{\expmone}{\texttt{expm1}}
\newcommand{\logonep}{\texttt{log1p}}
\newcommand{\smid}{\bar s}
\newcommand{\Utilde}{\tilde U}
\newcommand{\Btilde}{\widetilde \B}
\newcommand{\betatilde}{\tilde\beta}
\newcommand{\Children}{\operatorname{Children}}
\newcommand{\divideproc}{\textsl{divide}}
\newcommand{\free}{\textsl{free}}
\newcommand{\ANDAND}{\textbf{and}\xspace}
\newcommand{\Sigmatilde}{\widetilde\Sigma}
\newcommand{\vecU}{\boldsymbol{U}}

\newcommand{\Leaves}{\mathcal{L}}

\newcommand{\G}{\mathcal{G}}
\newcommand{\Parent}{\operatorname{Parent}}
\newcommand{\Gen}{\operatorname{Gen}}
\newcommand{\Len}{\operatorname{Len}}
\newcommand{\kmin}{k_{\min}}
\newcommand{\nmin}{n_{\min}}
\newcommand{\nmax}{n_{\max}}
\newtheorem{theorem}{Theorem}[section]
\newtheorem{corollary}[theorem]{Corollary}
\newtheorem{definition}[theorem]{Definition}
\newtheorem{lemma}[theorem]{Lemma}
\newtheorem{proposition}[theorem]{Proposition}
\title{Fast summation by interval clustering for an evolution 
equation with memory}
\author{William McLean}
\date{\today}
\begin{document}
\maketitle
\begin{abstract}
We solve a fractional diffusion equation using a piecewise-constant,
discontinuous Galerkin method in time combined with a continuous,
piecewise-linear finite element method in space.  If there are $N$~time
levels and $M$~spatial degrees of freedom, then a direct implementation
of this method requires $O(N^2M)$~operations and $O(NM)$~storage, owing to
the presence of a memory term: at each time step, the discrete evolution
equation involves a sum over \emph{all} previous time levels.  We show 
how the computational cost can be reduced to $O(MN\log N)$~operations 
and $O(M\log N)$~active memory locations.  
\end{abstract}
%%%%%%%%%%%%%%%%%%%%%%%%%%%%%%%%%%%%%%%%%%%%%%%%%%%%%%%%%%%%%%%%%%%%%
\section{Introduction}
The density~$u=u(x,t)$ of particles undergoing anomalous subdiffusion 
satisfies the integrodifferential equation~\cite{MetzlerKlafter2000}
\begin{equation}\label{eq: integrode}
\frac{\partial u}{\partial t}(x,t)
	-\nabla\cdot\biggl(\frac{\partial}{\partial t}
	\int_0^t\frac{(t-s)^{\nu-1}}{\Gamma(\nu)}\,K\nabla u(x,s)\,ds\biggr)
	=f(x,t)
\end{equation}
for a parameter~$\nu$ in the range $0<\nu<1$, where $K>0$ is a generalized
diffusivity and $f$ is a homogeneous term.  We 
consider~\eqref{eq: integrode} for~$0<t<T$ and for~$x$ in a bounded, 
convex~or $C^2$ domain~$\Omega\subseteq\R^d$, subject to 
homogeneous boundary conditions, either of Dirichlet type,
\begin{equation}\label{eq: Dirichlet}
u(x,t)=0\quad\text{for $x\in\partial\Omega$,}
\end{equation}
or of Neumann type,
\begin{equation}\label{eq: Neumann}
\frac{\partial u}{\partial\nrml}(x,t)=0\quad\text{for $x\in\partial\Omega$,}
\end{equation}
where $\nrml$ denotes the outward unit normal for~$\Omega$.  In addition,
we specify the initial condition
\[
u(x,0)=u_0(x)\quad\text{for $x\in\Omega$.}
\]
In the limit as~$\nu\to1$, the evolution equation~\eqref{eq: integrode}
reduces to the classical diffusion equation, $u_t-\nabla\cdot(K\nabla u)=f$,
in which the flux~$-K\nabla u$ depends only on the instantaneous value
of the gradient.  By contrast, in~\eqref{eq: integrode} the flux 
depends on the entire history of the gradient, and this fact leads to 
significant computational challenges, particularly if the spatial 
dimension~$d=3$.

In~\cite{McLeanMustapha2009}, we solved the foregoing initial-boundary 
value problem using a piecewise-constant, discontinuous Galerkin (DG) 
method for the time discretization, combined with a standard, continuous
piecewise-linear finite element discretization in space.
For simplicity, in the present work
we assume that both the spatial mesh and the time steps are quasiuniform.
Defining the elliptic partial differential 
operator~$Au=-\nabla\cdot(K\nabla u)$, we assume that for some
$\sigma\in(0,1]$ the 
solution~$u$ satisfies regularity estimates of the form~\cite{McLean2010} 
\begin{equation}\label{eq: regularity}
t^\nu\|Au'(t)\|+t^{\nu+1}\|Au''(t)\|\le Ct^{\sigma-1}
\quad\text{and}\quad
\|Au(t)\|+t\|Au'(t)\|\le C
\end{equation}
for $0<t\le T$.  The DG solution~$U(x,t)$ therefore satisfies an error 
bound~\cite[Theorem~3]{McLeanMustapha2009} in the norm of~$L_2(\Omega)$,
\begin{equation}\label{eq: DG error}
\|U^n-u(t_n)\|\le C(k^\sigma+h^2)\lg(t_n/t_1)
\quad\text{for $1\le m\le M$ and $1\le n\le N$,}
\end{equation}
where $t_n$ is the $n$th time level, $U^n(x)=U(x,t_n^-)$, 
$k$ is the maximum time step, $h$ 
is the maximum diameter of the spatial finite elements and
$\lg(s)=\max(1,|\log s|)$.  When $\sigma<1$, we could achieve full 
first-order accuracy with respect to~$k$ (ignoring logarithmic factors) 
by relaxing the assumption that the time steps are quasi-uniform, but to 
do so would complicate our fast solution procedure.  

Section~\ref{sec: DG method} provides a precise description of the DG 
method, which can be interpreted as a type of implicit Euler scheme.
Denote the $m$th \emph{free} node by~$x_m$, and let $U^n_m=U(x_m,t_n^-)$.
At the $n$th time step, we compute the vector of nodal values, 
$\vecU^n=[U^n_m]\in\R^M$, by solving a linear system 
\begin{equation}\label{eq: linear system}
(\Mass+\beta_{nn}\Stiff)\vecU^n=\Mass\vecU^{n-1}+k_n\bar\loadf^n
	+\Stiff\sum_{j=1}^{n-1}\beta_{nj}\vecU^j.
\end{equation}
Here, $\Mass$~and $\Stiff$ are the mass and stiffness matrices arising
from the spatial discretization, $k_n=t_n-t_{n-1}$ is the length of the
$n$th time interval, $\bar f^n$ is the average value of 
the load vector for~$t_{n-1}<t<t_n$, and the weights are given by
\begin{equation}\label{eq: beta_nn}
\beta_{nn}=\int_{t_{n-1}}^{t_n}\frac{(t_n-s)^{\nu-1}}{\Gamma(\nu)}\,ds
	=\frac{k_n^\nu}{\Gamma(1+\nu)}
\end{equation}
and
\begin{equation}\label{eq: beta_nj}
\beta_{nj}=\int_{t_{j-1}}^{t_j}\biggl(\frac{(t_{n-1}-s)^{\nu-1}}{\Gamma(\nu)}
	-\frac{(t_n-s)^{\nu-1}}{\Gamma(\nu)}\biggr)\,ds>0.
\end{equation}
The condition number of the matrix~$\Mass+\beta_{nn}\Stiff$ is 
$O(1+k^\nu h^{-2})$, 
and we assume the use of an efficient elliptic solver costing $O(M)$~operations.
By comparison, computing the right-hand side in the obvious way requires
$O(nM)$~operations.
Moreover, we must keep the vector~$U^j$ in active memory for \emph{all} the 
previous time levels $j=1$, $2$, \dots, $n-1$,
requiring $O(nM)$~locations.  Thus, the total cost for $N$~time steps 
is $O(N^2M)$~operations and $O(NM)$~active memory locations, whereas 
applying the same DG method to a \emph{classical} diffusion equation 
(which has no memory term) costs only $O(NM)$~operations and 
$O(M)$~active locations.  In other words, solving the fractional 
diffusion equation in this way costs $N$~times as 
much as solving a classical diffusion equation.

Cuesta, Lubich and Palencia~\cite{CuestaLubichPalencia2006} studied 
the time discretization of~\eqref{eq: integrode} by convolution quadrature,
and Sch\"adle, L\'opez-Fern\'andez and 
Lubich~\cite{SchaedleLopezFernandezLubich2006,LopezFernandezLubichSchaedle2008} 
developed a fast solution algorithm costing $O(MN\log N)$~operations
and using $O(M\log N)$~active memory locations.  
The purpose of this paper is to present a fast summation algorithm for
the DG method~\eqref{eq: linear system} that likewise costs 
$O(MN\log N)$~operations and $O(M\log N)$~active memory locations.

The algorithm is closely related to the panel clustering technique for
boundary element methods, introduced by Hackbusch and 
Nowak~\cite{HackbuschNowak1989}.  To explain the basic strategy, suppose that
instead of $(t-s)^{\nu-1}/\Gamma(\nu)$ the integral term had a degenerate 
kernel~$\sum_{p=1}^r a_p(t)b_p(s)$.  In this case,
\begin{equation}\label{eq: low rank}
\beta_{nj}=\sum_{p=1}^r\phi_{pn}\psi_{pj}
	\quad\text{for $1\le j\le n\le N$,}
\end{equation}
where $\phi_{np}=a_p(t_n)-a_p(t_{n-1})$~and 
$\psi_{nj}=\int_{t_{j-1}}^{t_j}b_p(s)\,ds$.  Since
\[
\sum_{j=1}^{n-1}\beta_{nj}\vecU^j=\sum_{p=1}^r\phi_{pn}
	\Psi^{n-1}_p(\vecU)
	\quad\text{where}\quad\Psi^{n-1}_p(\vecU)
	=\sum_{j=1}^{n-1}\psi_{pj}\vecU^j,
\]
and since $\Psi^n_p(\vecU)=\Psi^{n-1}_p(\vecU)+\psi_{pn}\vecU^n$, 
evaluating the right-hand side of the linear system~\eqref{eq: linear system}
would cost only $O(rM)$~operations and there would be no need to retain
in active memory the solution at all previous time levels.  Our 
kernel~$(t-s)^{\nu-1}/\Gamma(\nu)$ is not degenerate,
but it can be approximated to high accuracy by a degenerate kernel if we
restrict $t$~and $s$ to suitable, well-separated intervals.
Consequently, if $t_n$~and $t_j$ are restricted in the same way,
then $\beta_{nj}$ can be approximated to high accuracy by a 
sum~$\betatilde_{nj}$ of the form~\eqref{eq: low rank}, leading to a 
fast method to evaluate the sum that occurs on the right-hand side 
of~\eqref{eq: linear system}.

In Section~\ref{sec: DG method} we investigate the effect of perturbing 
the DG method in this way, and Section~\ref{sec: Low-rank} presents a
simple scheme for generating the $\betatilde_{nj}$ via Taylor expansion.
Section~\ref{sec: cluster} describes the \emph{cluster tree}, whose nodes are
contiguous families of time intervals, used to appropriately restrict 
$t_n$~and $t_j$.  The fast summation algorithm is
then defined via the concept of an \emph{admissible covering}, and we
present an error estimate in Theorem~\ref{thm: error}.  Further 
investigation of the cluster tree in Section~\ref{sec: cost} allows us 
to prove, in Theorem~\ref{thm: operations}, that the algorithm requires
$O(NM\log N)$~operations.  In Section~\ref{sec: memory}, we present
a memory management strategy and show, in Theorem~\ref{thm: active memory}, 
that with this strategy the fast summation algorithm uses at most 
$O(M\log N)$~active memory locations during each time step.
Section~\ref{sec: example} presents a numerical example, and the paper
concludes with three technical appendices.

Although presented here only for equation~\eqref{eq: integrode} 
discretized in time using a piecewise-constant DG method, the fast 
summation algorithm does not depend in any essential way on this specific
choice and could be used for many other time stepping procedures for
evolution problems with memory.  The key requirements are that the
quadrature weights are computed via local averages of the kernel and that, 
away from the diagonal, the derivatives of the kernel exist and decay
appropriately.  Also, the approximation scheme of Section~\ref{sec: Low-rank},
based on Taylor expansions, is only one possibility, chosen because it is
simple to analyse for the kernel in~\eqref{eq: integrode}.  We could 
instead use an interpolation scheme that requires a user to supply only
pointwise values of the kernel.
%%%%%%%%%%%%%%%%%%%%%%%%%%%%%%%%%%%%%%%%%%%%%%%%%%%%%%%%%%%%%%%%%%%%%
\section{Numerical method}\label{sec: DG method}
We define $\omega_\nu(t)=t^{\nu-1}/\Gamma(\nu)$ and denote
the Riemann--Liouville fractional differentiation operator of order~$1-\nu$
by
\[
\B v(t)=(\omega_\nu*v)_t=\frac{\partial}{\partial t}
	\int_0^t\omega_\nu(t-s)v(s)\,ds,
\]
and let $u_t=\partial u/\partial t$ so that, suppressing the dependence
on~$x$,
\begin{equation}\label{eq: ivp}
u_t+\B Au=f(t)\quad\text{for $0<t<T$, with $u(0)=u_0$.}
\end{equation}
We also denote the inner product in~$L_2(\Omega)$ and the bilinear form 
associated with~$A$ by
\[
\iprod{u,v}=\int_\Omega uv\,dx
\quad\text{and}\quad
\bilinA(u,v)=\int_\Omega K\nabla u\cdot\nabla v\,dx.
\]
The weak form of~\eqref{eq: ivp} is then
\[ 
\iprod{u_t,v}+\bilinA(\B u,v)=\iprod{f(t),v}
\quad\text{for all test functions~$v\in L_2\bigl((0,T),\Hdot^1\bigr)$,} 
\]
where $\Hdot^1$ is the Sobolev space~$H^1_0(\Omega)$ in the case of Dirichlet 
boundary conditions~\eqref{eq: Dirichlet}, and $H^1(\Omega)$ in the
case of  Neumann boundary conditions~\eqref{eq: Neumann}.

The numerical solution~$U(t)\approx u(t)$ is a piecewise-constant 
function of~$t$, with coefficients belonging to a continuous, 
piecewise-linear finite element space~$S_h\subseteq\Hdot^1$.  Thus, $U$ 
generally has a jump discontinuity at each time level~$t_n$,  and we 
adopt the usual convention of treating $U(t)$ as continuous for~$t$ 
in the half-open interval~$I_n=(t_{n-1},t_n]$, writing
\[
U^n=U(t_n)=U(t_n^-),\qquad
U^n_+=U(t_n^+),\qquad [U]^n=U^n_+-U^n.
\]
The DG time-stepping procedure is determined by the variational equation
\begin{multline*}
\iprod{U^{n-1}_+,V^{n-1}_+}+\int_{I_n}\bigl[
	\iprod{U'(t),V(t)}+A\bigl(\B U(t),V(t)\bigr)\bigr]\,dt\\
	=\iprod{U^{n-1},V^{n-1}_+}+\int_{I_n}\iprod{f(t),V(t)}\,dt,
\end{multline*}
which must hold for every piecewise-constant test function~$V$ 
with coefficients in~$S_h$.  In our piecewise-constant
case, $U^{n-1}_+=U^n$ so this variational equation reduces to
finding $U^n\in S_h$ such that
\[
\biggl\langle\frac{U^n-U^{n-1}}{k_n},\chi\biggr\rangle
	+\bilinA\bigl(\Bbar^n U,\chi\bigr)=\iprod{\bar f^n,\chi}
	\quad\text{for all $\chi\in S_h$,}
\]
where $\bar f^n=k_n^{-1}\int_{I_n}f(t)\,dt$ and
\begin{equation}\label{eq: Bbar}
\Bbar^n U=\frac{1}{k_n}\int_{I_n}\B U(t)\,dt
	=\frac{1}{k_n}\biggl(\beta_{nn}U^n
	-\sum_{j=1}^{n-1}\beta_{nj}U^j\biggr).
\end{equation}
Thus, the vector of nodal values of~$U^n$ 
satisfies~\eqref{eq: linear system}, with weights given by
\eqref{eq: beta_nn}~and \eqref{eq: beta_nj}.

Our fast algorithm approximates $\Bbar^nU$ with
\begin{equation}\label{eq: Btilde}
\Btilde^n U=\frac{1}{k_n}\biggl(\beta_{nn}U^n
	-\sum_{j=1}^{n-1}\betatilde_{nj}U^j\biggr),
\end{equation}
where $\betatilde_{nj}\approx\beta_{nj}$ for $1\le j\le n-1$, and yields
$\Utilde^n$ satisfying the perturbed equation
\begin{equation}\label{eq: perturbed}
\biggl\langle\frac{\Utilde^n-\Utilde^{n-1}}{k_n},\chi\biggr\rangle
	+\bilinA\bigl(\Btilde^n\Utilde,\chi\bigr)=\iprod{\bar f^n,\chi}
	\quad\text{for all $\chi\in S_h$,}
\end{equation}
with $\Utilde^0=U^0$.  

\begin{theorem}
If, for every real-valued, piecewise-constant function~$V$,
\begin{equation}\label{eq: pos}
\sum_{n=1}^N k_n(\Btilde^n V)V^n\ge 0,
\end{equation}
then the perturbed problem~\eqref{eq: perturbed} is stable:
\[
\|\Utilde^n\|\le\|\Utilde^0\|+2\sum_{j=1}^n k_j\|\bar f^j\|
	\quad\text{for $1\le n\le N$.}
\]
\end{theorem}

\begin{proof}
This estimate follows by a simple energy 
argument~\cite[Theorem~1]{McLeanMustapha2009}.
\end{proof}

In Appendix~\ref{sec: lower bound}, we prove a lower bound
\begin{equation}\label{eq: lower}
\sum_{n=1}^N k_n(\Bbar^n V)V^n\ge \rho_\nu T^{\nu-1}\sum_{n=1}^N k_n(V^n)^2,
\end{equation}
where the constant~$\rho_\nu>0$ depends only on~$\nu$, allowing us to
show the following.

\begin{corollary}\label{cor: stability}
The perturbed DG method~\eqref{eq: perturbed} is stable if
\[
\sum_{j=1}^{n-1}|\betatilde_{nj}-\beta_{nj}|\le\rho_\nu T^{\nu-1}k_n
\quad\text{for $2\le n\le N$,}
\]
and
\[
\sum_{n=j+1}^N|\betatilde_{nj}-\beta_{nj}|\le\rho_\nu T^{\nu-1}k_j
\quad\text{for $1\le j\le N-1$.}
\]
\end{corollary}
\begin{proof}
Write
\[
\Delta=\sum_{n=1}^N k_n(\Btilde^n V)V^n-\sum_{n=1}^N k_n(\Bbar^n V)V_n
	=\sum_{n=2}^N\sum_{j=1}^{n-1}(\beta_{nj}-\betatilde_{nj})V_jV_n,
\]
and observe that since~$|V^jV^n|\le\tfrac12(V^j)^2+\tfrac12(V^n)^2$,
\begin{multline*}
|\Delta|\le\frac12\sum_{j=1}^{N-1}(V^j)^2\sum_{n=j+1}^N
		|\betatilde_{nj}-\beta_{nj}|
	+\frac12\sum_{n=2}^N(V^n)^2\sum_{j=1}^{n-1}
		|\betatilde_{nj}-\beta_{nj}|\\
	\le\rho_\nu T^{\nu-1}\sum_{n=1}^Nk_n(V^n)^2.
\end{multline*}
\end{proof}

To estimate the effect on the solution~$U^n$ of perturbing 
the weights~$\beta_{nj}$, we introduce the piecewise-constant interpolation
operator
\[
\Pi u(t)=u(t_n)\quad\text{for $t_{n-1}<t\le t_n$.}
\]
In the next result, for simplicity we treat the semi-discrete DG method
in which there is no spatial discretization; the fully-discrete method
could be analysed using the methods of McLean and
Mustapha~\cite[Section~5]{McLeanMustapha2009}.

\begin{theorem}\label{thm: accuracy}
Assume that $S_h=\Hdot^1$ and $\Utilde^0=u_0$.
If the perturbed equation~\eqref{eq: perturbed} is stable,  and if
\[
\sum_{j=1}^n\sum_{l=1}^{j-1}|\betatilde_{jl}-\beta_{jl}|
	\le\epsilon_n,
\]
then
\begin{equation}\label{eq: time error}
\|\Utilde^n-u(t_n)\|\le2\epsilon_n\max_{1\le j\le n-1}\|Au(t_j)\|
	+2\sum_{j=1}^nk_j\|\Bbar^n A(u-\Pi u)\|.
\end{equation}
\end{theorem}
\begin{proof}
Since $\Utilde$ satisfies~\eqref{eq: perturbed} and $u$ satisfies
\[
\biggl\langle\frac{u(t_n)-u(t_{n-1})}{k_n},\chi\biggr\rangle
	+\bilinA\bigl(\Bbar^n u,\chi\bigr)=\iprod{\bar f^n,\chi}
	\quad\text{for all $\chi\in S_h$,}
\]
the difference~$W^n=\Utilde^n-U^n$ satisfies 
\[
\biggl\langle\frac{W^n-W^{n-1}}{k_n},\chi\biggr\rangle
	+\bilinA\bigl(\Btilde^n W,\chi\bigr)=\iprod{\bar g^n,\chi}
	\quad\text{for all $\chi\in S_h$,}
\]
where $\bar g^n=\bigl(\Bbar^n-\Btilde^n\bigr)\Pi Au+\Bbar^nA(u-\Pi u)$.
Noting that $W^0=0$, stability of the perturbed equation implies that
\[
\|W^n\|\le2\sum_{j=1}^nk_j\|\bar g^n\|,
\]
and from~\eqref{eq: Bbar}~and \eqref{eq: Btilde} we see that
\[
\bigl(\Bbar^j-\Btilde^j\bigr)\Pi Au
	=k_j^{-1}\sum_{l=1}^{j-1}
	\bigl(\betatilde_{jl}-\beta_{jl}\bigr)Au(t_j).
\]
\end{proof}

The second term on the right-hand side of~\eqref{eq: time error} does 
not involve~$\betatilde_{nj}$
and is $O(k^\sigma)$~\cite{McLeanMustapha2009}.
%%%%%%%%%%%%%%%%%%%%%%%%%%%%%%%%%%%%%%%%%%%%%%%%%%%%%%%%%%%%%%%%%%%%%
\section{Taylor approximation of the weights}
\label{sec: Low-rank}
Recall from~\eqref{eq: beta_nn} that $\beta_{nn}=\omega_{1+\nu}(k_n)$,
and from~\eqref{eq: beta_nj} that for $1\le j\le n-1$,
\begin{equation}\label{eq: beta_nj repn}
\beta_{nj}=\int_{I_j}\bigl[\omega_\nu(t_{n-1}-s)-\omega_\nu(t_n-s)\bigr]
	\,ds
	=-\int_{I_j}\int_{I_n}\omega_{\nu-1}(t-s)\,dt\,ds.
\end{equation}
Denote the midpoint of the interval~$I_n$ by 
$t_{n-1/2}=\tfrac12(t_{n-1}+t_n)$, and define
\[
B_\mu(t,k)=\int_{-k/2}^{k/2}\omega_\mu(t+s)\,ds
	=\omega_{1+\mu}(t+\tfrac12 k)-\omega_{1+\mu}(t-\tfrac12 k)
\]
for $t>k/2$ and $-\infty<\mu<\infty$.
The approximate weights~$\betatilde_{nj}$ are determined as follows.
\ifpreprint
See Appendices \ref{sec: eval}~and \ref{sec: computing low} for notes
on the stable evaluation of these quantities.
\fi

\begin{theorem}\label{thm: Taylor expansion}
Let $0<\eta\le1$ and suppose that $0\le a<b<c<d\le T$ with
\begin{equation}\label{eq: a b c d admissible}
I_j\subseteq(a,b],\qquad I_n\subseteq(c,d],\qquad
\frac{b-a}{c-b}\le\eta.
\end{equation}
Denote the midpoint of~$[a,b]$ by $\smid=\tfrac12(a+b)$, and define 
\[
\phi_{pn}=(-1)^pB_{\nu-p}(t_{n-1/2}-\smid,k_n),\qquad
\psi_{pj}=B_p(t_{j-1/2}-\smid,k_j),
\]
and
\[
\betatilde_{nj}=\sum_{p=1}^r\phi_{pn}\psi_{pj}.
\]
Then
\[
|\betatilde_{nj}-\beta_{nj}|\le 2^{2-\nu}(r+1)(\eta/2)^r\beta_{nj}.
\]
\end{theorem}
\begin{proof}
Taylor expansion about~$\smid$ gives
\begin{multline*}
\omega_{\nu-1}(t-s)=\sum_{p=0}^{r-1}(-1)^p\omega_{\nu-1-p}(t-\smid)
	\omega_{1+p}(s-\smid)\\
	+(-1)^r\int_{\smid}^s\omega_r(s-y)\omega_{\nu-1-r}(t-y)\,dy,
\end{multline*}
so, by~\eqref{eq: beta_nj repn}, integrating over $t\in I_n$~and
$s\in I_j$ gives $\beta_{nj}=\betatilde_{nj}+E_{rnj}$ with
\[
E_{rnj}=(-1)^{r+1}\int_{I_n}\int_{I_j}\int_{\smid}^s\omega_r(s-y)
	\omega_{\nu-1-r}(t-y)\,dy\,ds\,dt.
\]
Since $y\in I_j\subseteq(a,b]$ and $t\in I_n\subseteq(c,d]$, we have 
$t-y\ge c-b$, so
\[
|E_{rnj}|\le\frac{(c-b)^{-r}}{|\Gamma(\nu-r-1)|}
	\int_{I_n}\int_{I_j}\frac{|s-\smid|^{r-1}}{r!}
	\biggl| \int_{\smid}^s\frac{dy}{(t-y)^{2-\nu}}\biggr|
	\,ds\,dt.
\]
If $\smid\le s$, then $y\le s$ so $t-y\ge t-s$. However, if $s\le\smid$
then $0\le y-s\le\smid-s\le\tfrac12(b-a)$ so
\[
t-y=(t-s)-(y-s)\ge(t-s)-\tfrac12(b-a)
	=(t-s)\biggl(1-\frac{\tfrac12(b-a)}{t-s}\biggr).
\]
Also, $t-s\ge c-b$ because $s\in[a,b]$, and 
assumption~\eqref{eq: a b c d admissible} implies that
$t-y\ge(t-s)(1-\tfrac12\eta)$.  Thus,
\[
\biggl| \int_{\smid}^s\frac{dy}{(t-y)^{2-\nu}}\,dy\biggr|
	\le\frac{1}{(1-\tfrac12\eta)^{2-\nu}}\,
	\frac{|s-\smid|}{(t-s)^{2-\nu}},
\]
and since $|s-\smid|\le\tfrac12(b-a)\le\tfrac12\eta(c-b)$, it follows that
\[
|E_{rnj}|\le\biggl(\frac{2}{2-\eta}\biggr)^{2-\nu}
	\frac{(\eta/2)^r}{r!|\Gamma(\nu-r-1)|}
	\int_{I_n}\int_{I_j}\frac{ds\,dt}{(t-s)^{2-\nu}}.
\]
The identity 
$\Gamma(\nu-1)=(\nu-2)(\nu-3)\cdots(\nu-r)(\nu-r-1)\Gamma(\nu-r-1)$
shows that
\[
\biggl|\frac{\Gamma(\nu-1)}{r!\Gamma(\nu-r-1)}\biggr|\le
	\frac{2-\nu}{2}\,\frac{3-\nu}{3}\cdots\frac{r-\nu}{r}\,(r+1-\nu)
	\le r+1,
\]
and because $\Gamma(\nu-1)=\Gamma(\nu)/(\nu-1)<0$,
\[
|E_{rnj}|\le 2^{2-\nu}(r+1)(\eta/2)^r\biggl(-\int_{I_n}\int_{I_j}
	\frac{(t-s)^{\nu-2}}{\Gamma(\nu-1)}\,ds\,dt\biggr).
\]
\end{proof}

\ifpreprint
Whereas $\psi_{pj}$ depends on~$[a,b]$ (via $\smid$), in the
following alternative expansion $\psi^\star_{pj}$ depends only on
$p$~and $j$.  However, the sum that defines $\phi_{pn}^\star$ is
susceptible to loss of precision.

\begin{corollary}\label{cor: phi psi star}
If we define
\[
\phi_{pn}^\star=\sum_{q=p}^r\frac{(-\smid)^{q-p}}{(q-p)!}
	\,\phi_{qn}
\quad\text{and}\quad
\psi_{pj}^\star=B_p(t_{j-1/2},k_j),
\]
then
\[
\betatilde_{nj}=\sum_{p=0}^r\phi^\star_{pn}\psi^\star_{pj}.
\]
\end{corollary}
\begin{proof}
The binomial expansion gives
\[
\psi_{pj}=\int_{I_j}\frac{(s-\smid)^{p-1}}{(p-1)!}\,ds
	=\sum_{q=1}^p\frac{(-\smid)^{p-q}}{(p-q)!}
	\int_{I_j}\frac{s^{q-1}}{(q-1)!}\,ds
\]
so
\[
\betatilde_{nj}=\sum_{q=1}^r\biggl(\sum_{p=q}^r\phi_{pn}\,
	\frac{(-\smid)^{p-q}}{(p-q)!}\biggr)
	\int_{I_j}\frac{s^{q-1}}{(q-1)!}\,ds.
\]
\end{proof}
\fi
%%%%%%%%%%%%%%%%%%%%%%%%%%%%%%%%%%%%%%%%%%%%%%%%%%%%%%%%%%%%%%%%%%%%%
\section{Cluster tree}
\label{sec: cluster}
We introduce the notation
\[
\Cluster(j,n)=\{I_j,I_{j+1},\ldots,I_n\}\quad\text{for $1\le j\le n\le N$,}
\]
and refer to any such set of consecutive subintervals as a~\emph{cluster}.
A \emph{cluster tree} for the mesh~$0=t_0<t_1<\cdots<t_N=T$ is a 
tree~$\Tree$ having the following properties:
\begin{enumerate}
\item each node of~$\Tree$ is a cluster;
\item the root node of~$\Tree$ is $\Cluster(1,N)$;
\item each node is either a leaf or else equals the disjoint union of
its children.
\end{enumerate}
Let $\Leaves$ denote the set of leaves in the cluster tree~$\Tree$, and
observe that for any two distinct nodes $\Cluster_1\ne\Cluster_2$ of~$\Tree$ 
exactly one of the following holds~\cite[Remark~3.4]{HackbuschNowak1989}: 
$\Cluster_1\subsetneq\Cluster_2$; $\Cluster_2\subsetneq\Cluster_1$; 
or $\Cluster_1\cap\Cluster_2=\emptyset$.

In the obvious way, we define the \emph{length} of a cluster~$\Cluster$ 
to be the length of the underlying interval~$\bigcup\Cluster$;
thus,
\[
\Len(\Cluster)=t_n-t_{j-1}\quad\text{if $\Cluster=\Cluster(j,n)$.}
\]
The \emph{distance} between two clusters is the Euclidean distance between 
the underlying point sets in~$\R$,
\[
\Dist(\Cluster_1,\Cluster_2)=\inf\{\,|t-s|:
	\text{$t\in\textstyle{\bigcup\Cluster_1}$ and 
              $s\in\textstyle{\bigcup\Cluster_2}$}\,\},
\]
so in particular, 
$\Dist\bigl(\Cluster(j_1,n_1),\Cluster(j_2,n_2)\bigr)=t_{j_2-1}-t_{n_1}$
if $j_2>n_1$.
We define the \emph{history} of a cluster to be the entire preceding
half-open interval:
\[
\History(\Cluster)=(0,t_{j-1}]\quad\text{if $\Cluster=\Cluster(j,n)$.}
\]

Given a leaf~$\Leaf\in\Leaves$, and a parameter~$\eta$ in the
range~$0<\eta\le1$, we say that a cluster~$\Cluster$ 
is $(\Leaf,\eta)$-\emph{admissable} if
\[
\bigcup\Cluster\subseteq\History(\Leaf)\quad\text{and}\quad
\Len(\Cluster)\le\eta\Dist(\Cluster,\Leaf).
\]
Thus, the conclusions of Theorem~\ref{thm: Taylor expansion}
hold for $I_j\in\Cluster$ and $I_n\in\Leaf$ with
$\bigcup\Cluster=(a,b]$~and $\bigcup\Leaf=(c,d]$.
Notice that if $\Cluster$ is $(\Leaf,\eta)$-admissible, then so are all
the descendents of~$\Cluster$.

An \emph{$(\Leaf,\eta)$-admissible cover} is a set~$\Cover$ 
of clusters such that
\begin{enumerate}
\item each cluster~$\Cluster\in\Cover$ is a node of~$\Tree$;
\item $\bigcup\{\,I:I\in\Cluster\in\Cover\,\}=\History(\Leaf)$;
\item if $\Cluster_1$, $\Cluster_2\in\Cover$ with~$\Cluster_1\ne\Cluster_2$,
then $\Cluster_1\cap\Cluster_2=\emptyset$;
\item if $\Cluster\in\Cover$ then either $\Cluster$ is 
$(\Leaf,\eta)$-admissible, or else $\Cluster\in\Leaves$.
\end{enumerate}
A trivial example of an $(\Leaf,\eta)$-admissible cover is the set of
leaves in the history of~$\Leaf$, that is,
\[
\Cover=\{\,\Leaf'\in\Leaves:\Leaf'\subseteq\History(\Leaf)\}.
\]

\begin{lemma}\label{lem: non-ad leaves}
All $(\Leaf,\eta)$-admissible covers contain the same non-admissible 
leaves.
\end{lemma}
\begin{proof}
Let $\Cover_1$ and $\Cover_2$ be $(\Leaf,\eta)$-admissible covers, and
suppose that $\Leaf'\in\Cover_1\cap\Leaves$ is not $(\Leaf,\eta)$-admissible.
Since 
$\Leaf'\subseteq\History(\Leaf)=\bigcup\{\,I: I\in\Cluster\in\Cover_2\,\}$,
there exist $I\in\Cluster\in\Cover_2$ such that $I$ intersects at least
one interval~$I'\in\Leaf'$.  Since $\Leaf'$ is a leaf of~$\Tree$, it follows
that $\Leaf'\subseteq\Cluster$ and hence $\Cluster$ is not 
$(\Leaf,\eta)$-admissible.  Thus, $\Cluster$ must be a leaf, because
$\Cover_2$ is an $(\Leaf,\eta)$-admissible cover, and we conclude that 
$\Cluster=\Leaf'$.
\end{proof}

Algorithm~\ref{alg: divide} defines a recursive procedure, 
$\divideproc(\Cluster,\Cover,\Leaf,\eta)$, that we use to define
a family of clusters~$\MinC(\Leaf)$, as follows:
\begin{algorithmic}
\STATE $\Cluster=\Cluster(1,N)$, $\Cover=\emptyset$
\STATE $\divideproc(\Cluster, \Cover, \Leaf, \eta)$
\STATE $\MinC(\Leaf)=\Cover$
\end{algorithmic} 
This construction is a modified version of~\cite[(3.8b)]{HackbuschNowak1989}.

\begin{theorem}\label{thm: MinC}
$\MinC(\Leaf)$ is the unique minimal $(\Leaf,\eta)$-admissible cover.
\end{theorem}
\begin{proof}
Suppose that $\Cover$ is an $(\Leaf,\eta)$-admissible cover.
Let $\Mluster\in\MinC(\Leaf)$ and let $\Cluster$ be any node of~$\Cover$ 
that intersects $\Mluster$.  If $\Mluster$ 
is a leaf, then $\Mluster$ is not $(L,\eta)$-admissible so 
$\Mluster=\Cluster$ by Lemma~\ref{lem: non-ad leaves}.  If $\Mluster$
is not a leaf, then either $\Cluster$ is a leaf, in which 
case $\Cluster\subsetneq\Mluster$, or else $\Cluster$ is
$(\Leaf,\eta)$-admissible and $\Cluster\subseteq\Mluster$ by the 
construction of~$\MinC(\Leaf)$ using Algorithm~\ref{alg: divide}.  
Hence, in all cases $\Cluster\subseteq\Mluster$, so either
$\Cover=\MinC(\Leaf)$ or else $\#\MinC(\Leaf)<\#\Cover$.
\end{proof}

\begin{algorithm}
\caption{$\divideproc(\Cluster,\Cover,\Leaf,\eta)$}
\label{alg: divide}
\begin{algorithmic}
\STATE Determine $a$, $b$, $c$, $d$ such that $\bigcup\Cluster=(a,b]$ and 
$\bigcup\Leaf=(c,d]$
\STATE $\Cover=\emptyset$
\IF{$a\le c$}
  \IF[$\Cluster$ is $(\Leaf,\eta)$-admissible]{$b\le c$ \ANDAND $b-a\le\eta(c-b)$}
    \STATE $\Cover=\Cover\cup\{\Cluster\}$
  \ELSIF{$b\le c$ \ANDAND $\Cluster\in\Leaves$}
    \STATE $\Cover=\Cover\cup\{\Cluster\}$
  \ELSE
    \FORALL{$\Cluster'\in\Children(\Cluster)$}
      \STATE $\divideproc(\Cluster',\Cover,\Leaf,\eta)$
    \ENDFOR
  \ENDIF
\ENDIF
\end{algorithmic}
\end{algorithm}

The minimal admissible cover~$\MinC(\Leaf)$ is the disjoint union of the 
sets
\begin{equation}\label{eq: Near Far}
\begin{aligned}
\Near(\Leaf)&=\{\,\Cluster\in\MinC(\Leaf):
\text{$\Cluster$ is not $(\Leaf,\eta)$-admissible}\,\},\\
\Far(\Leaf)&=\{\,\Cluster\in\MinC(\Leaf):
\text{$\Cluster$ is $(\Leaf,\eta)$-admissible}\,\}.
\end{aligned}
\end{equation}
Given $n$, there is a unique leaf~$\Leaf=\Leaf_n$ containing $I_n$,
and we define $\phi_{pn}(\Cluster)$
and $\psi_{pn}(\Cluster)$ using the formulae of 
Theorem~\ref{thm: Taylor expansion} with $(a,b]=\bigcup\Cluster$
and $(c,d]=\bigcup\Leaf_n$.  We then define
\begin{equation}\label{eq: beta near}
\betatilde_{nj}=\beta_{nj}\quad
	\text{if $I_j\in\Leaf_n$ or $I_j\in\Cluster\in\Near(\Leaf_n)$,}
\end{equation}
and
\begin{equation}\label{eq: beta far}
\betatilde_{nj}=\sum_{p=1}^r\phi_{pn}(\Cluster)\psi_{pj}(\Cluster)\quad
	\text{if $I_j\in\Cluster\in\Far(\Leaf_n)$,}
\end{equation}
so that
\begin{equation}\label{eq: fastsum}
\sum_{j=1}^{n-1}\betatilde_{nj}V^j=\Sigma_n(\Leaf_n,V)
	+\sum_{\Cluster\in\Near(\Leaf_n)}\Sigma_n(\Cluster,V)
	+\sum_{\Cluster\in\Far(\Leaf_n)}\Sigmatilde_n(\Cluster,V),
\end{equation}
where
\[
\Sigma_n(\Cluster,V)=\sum_{\substack{I_j\in\Cluster\\ j\le n-1}}
	\beta_{nj}V^j
\quad\text{and}\quad
\Sigmatilde_n(\Cluster,V)=\sum_{I_j\in\Cluster}\betatilde_{nj}V^j.
\]
To evaluate $\Sigmatilde_n(\Cluster,V)$, we compute
\[
\Sigmatilde_n(\Cluster,V)=\sum_{p=1}^r\phi_{pn}(\Cluster)\Psi_p(\Cluster,V)
\quad\text{where}\quad
\Psi_p(\Cluster,V)=\sum_{I_j\in\Cluster}\psi_{pj}(\Cluster) V^j.
\]

The results of Sections \ref{sec: DG method}~and \ref{sec: Low-rank} 
now yield the following estimate for the additional error incurred by 
using the approximating sum~\eqref{eq: fastsum}.  Recall that $\rho_\nu$ 
is the constant appearing in the lower bound~\eqref{eq: lower}.

\begin{theorem}\label{thm: error}
Define $\betatilde_{nj}$ according to Theorem~\ref{thm: Taylor expansion}, 
\eqref{eq: beta near}~and \eqref{eq: beta far}, and put 
$\kmin=\min_{1\le n\le N}k_n$.
The perturbed problem~\eqref{eq: perturbed} is stable if
\[
(r+1)(\eta/2)^r\le2^{\nu-2}\Gamma(\nu+1)\rho_\nu(\kmin/T)^{1-\nu},
\]
in which case the error estimate~\eqref{eq: time error} for the
semidiscrete DG method holds with
\[
\epsilon_n=\frac{2^{2-\nu}}{\Gamma(\nu+1)}\,(r+1)(\eta/2)^r nk^\nu
	\quad\text{for $1\le n\le N$.}
\]
\end{theorem}
\begin{proof}
Recalling \eqref{eq: beta_nj}, we have
\begin{align*}
\sum_{j=1}^{n-1}\beta_{nj}&=\int_0^{t_{n-1}}\biggl(
	\frac{(t_{n-1}-s)^{\nu-1}}{\Gamma(\nu)}
	-\frac{(t_n-s)^{\nu-1}}{\Gamma(\nu)}\biggr)\,ds\\
	&=\frac{k_n^\nu-\bigl(t_n^\nu-t_{n-1}^\nu\bigr)}{\Gamma(\nu+1)}
	\le\frac{k_n^\nu}{\Gamma(\nu+1)}
\end{align*}
and
\begin{align*}
\sum_{n=j+1}^N\beta_{nj}&=\int_{t_{j-1}}^{t_j}
	\frac{(t_j-s)^{\nu-1}-(T-s)^{\nu-1}}{\Gamma(\nu)}\,ds\\
	&=\frac{k_j^\nu-[(T-t_{j-1})^\nu-(T-t_j)^\nu]}{\Gamma(\nu+1)}
	\le\frac{k_j^\nu}{\Gamma(\nu+1)}.
\end{align*}
Thus, by Theorem~\ref{thm: Taylor expansion},
\[
\sum_{j=1}^{n-1}\bigl|\betatilde_{nj}-\beta_{nj}\bigr|
	\le2^{2-\nu}(r+1)(\eta/2)^r\,\frac{k_n^\nu}{\Gamma(\nu+1)}
\]
and
\[
\sum_{n=j+1}^N\bigl|\betatilde_{nj}-\beta_{nj}\bigr|
	\le2^{2-\nu}(r+1)(\eta/2)^r\,\frac{k_j^\nu}{\Gamma(\nu+1)}.
\]
Therefore, Corollary~\ref{cor: stability} shows that the perturbed
scheme is stable if
\[
2^{2-\nu}(r+1)(\eta/2)^r\,\frac{k_n^\nu}{\Gamma(\nu+1)}
	\le\rho_\nu T^{\nu-1}k_n\quad\text{for $1\le n\le N$,}
\]
and since
\[
\sum_{j=1}^n\sum_{l=1}^{j-1}|\betatilde_{nj}-\beta_nj|\le2^{2-\nu}
	(r+1)(\eta/2)^r\sum_{j=1}^n\frac{k_j^\nu}{\Gamma(\nu+1)},
\]
the error bound follows at once from Theorem~\ref{thm: accuracy}.
\end{proof}

Since $N^{-1}\le C\kmin/T$ because the mesh is quasiuniform, we have
stability if $(r+1)(\eta/2)^r\le CN^{\nu-1}$, 
and the error factor satisfies
\begin{equation}\label{eq: consistent error}
\epsilon_n=O(k)\quad\text{if}\quad
(r+1)(\eta/2)^r\le CN^{\nu-2}.
\end{equation}

\ifpreprint
The alternative expansion from Corollary~\ref{cor: phi psi star} gives
$\phi_{pn}^\star(\Cluster)$~and $\psi_{pj}^\star$ such that
\[
\betatilde_{nj}=\sum_{p=1}^r\phi_{pn}^\star(\Cluster)\psi_{pj}^\star\quad
	\text{if $I_j\in\Cluster\in\Far(\Leaf_n)$,}
\]
so
\[
\Sigmatilde_n(\Cluster,V)=\sum_{p=1}^r\phi_{pn}^\star(\Cluster)
	\Psi_p^\star(\Cluster,V)
\quad\text{where}\quad
\Psi_p^\star(\Cluster,V)=\sum_{I_j\in\Cluster}\psi_{pj}^\star V^j.
\]
If $\Cluster$ is not a leaf, then $\Cluster$ is the union of its
children, so
\[
\Psi_p^\star(\Cluster,V)=\sum_{\Cluster'\in\Children(\Cluster)}
	\Psi_p^\star(\Cluster',V).
\]
Thus, once we compute $\Psi_p^\star(\Cluster,V)$ for every
leaf~$\Cluster$ of~$\Tree$, we can compute $\Psi_p^\star(\Cluster,V)$
for the remaining nodes~$\Cluster$ by aggregation.
\fi
%%%%%%%%%%%%%%%%%%%%%%%%%%%%%%%%%%%%%%%%%%%%%%%%%%%%%%%%%%%%%%%%%%%%%
\section{Computational cost}
\label{sec: cost}
We now seek to estimate the number of operations required to evaluate the
right-hand side of~\eqref{eq: fastsum}.  Let $\Gen(\Cluster)$ denote
the \emph{generation} of the node~$\Cluster\in\Tree$, defined recursively by 
\[
\Gen\bigl(\Cluster)=0\quad\text{if $\Cluster=\Cluster(1,N)$,}
\]
and
\[
\Gen(\Cluster)=\Gen\bigl(\Parent(\Cluster)\bigr)+1\quad
	\text{if $\Cluster\ne\Cluster(1,N)$.}
\]
Put
\[
\MinC_\ell(\Leaf)=\{\,\Cluster\in\MinC(\Leaf):\Gen(\Cluster)=\ell\,\},
\]
and note the $\MinC_0(\Leaf)=\emptyset$ since the root node of
the cluster tree cannot be $(\Leaf,\eta)$-admissible.
We formlate the following regularity condition for the cluster tree.

\begin{definition}
For integers $G\ge1$ and $Q\ge2$, we say that $\Tree$ is 
\emph{$(G,Q)$-uniform} if, for some constants $0<\lambda<\Lambda$ and
for every node~$\Cluster\in\Tree$,
\begin{enumerate}
\item
$0\le\Gen(\Cluster)\le G$;
\item
$\#\Children(\Cluster)=Q$ whenever $\Cluster\notin\Leaves$;
\item
$\lambda TQ^{-\ell}\le\Len(\Cluster)\le\Lambda TQ^{-\ell}$
whenever $\Cluster\in\G_\ell$;
\item
$\bigcup\G_\ell=[0,T]$ for $0\le\ell\le G$.
\end{enumerate}
\end{definition}

For example, given a uniform mesh with $N=2^P$~subintervals and $G\le P$, 
recursive bisection of~$[0,T]$ for $G$~generations leads to a 
$(G,2)$-uniform cluster tree in which each leaf contains 
$N/2^G=2^{P-G}$~subintervals.

We assume henceforth that $\Tree$ is $(G,Q)$-uniform.
Since $\#\G_0=1$ and $\#\G_{\ell+1}=Q\times\#\G_\ell$, we see that
$\#\G_\ell=Q^\ell$, implying that
\begin{equation}\label{eq: number of nodes}
\#\Tree=\sum_{\ell=0}^GQ^\ell=\frac{Q^{G+1}-1}{Q-1}\le2(\#\Leaves)-1.
\end{equation}
Also, $\Leaves=\G_G$ so
\[
Q^G=\#\G_G=\#\Leaves.
\]
The next result shows that $\#\MinC(\Leaf)=O(\eta^{-1}QG)$.

\begin{lemma}\label{lem: MinC count}
Suppose that $\Leaf\in\Leaves$ and $\Cluster\in\MinC_\ell(\Leaf)$.
If $\bigcup\Cluster=(a,b]$ and $\bigcup\Leaf=(c,d]$, then
\[
c-(1+\eta^{-1})\Lambda TQ^{-\ell+1}<a<b\le c -\eta^{-1}\lambda TQ^{-\ell}
\quad\text{when $1\le\ell\le G-1$,}
\]
whereas
\[
c-(1+\eta^{-1})\Lambda TQ^{-G}<a<b\le c
\quad\text{when $\ell=G$.}
\]
Therefore,
\[
\#\MinC_\ell(\Leaf)\le\frac{\Lambda}{\lambda}\times\begin{cases}
	(1+\eta^{-1})Q&\text{for $0\le\ell\le G$,}\\
	(1+\eta^{-1})&\text{for $\ell=G$.}
\end{cases}
\]
\end{lemma}
\begin{proof}
If $1\le\ell\le G-1$, then $\Cluster$ is $(\Leaf,\eta)$-admissible so
\[
\lambda TQ^{-\ell}\le\Len(\Cluster)=b-a\le\eta\Dist(\Cluster,\Leaf)
	=\eta(c-b)
\]
and thus $b\le c-\eta^{-1}\lambda TQ^{-\ell}$.  Suppose for a contradiction
that $a\le c-\Lambda T(1+\eta^{-1})Q^{-\ell+1}$, and let
$\Cluster'=\Parent(\Cluster)$.  If $\bigcup\Cluster'=(a',b']$, then
$a'\le a<b\le b'$ and $\Gen(\Cluster')=\ell-1$, so
\[
b'\le a'+\Len(\Cluster')\le a+\Lambda TQ^{-\ell+1}
	\le c-\eta^{-1}\Lambda TQ^{-\ell+1}
\]
and
\[
\Len(\Cluster')\le\Lambda TQ^{-\ell+1}\le\eta(c-b')
	=\eta\Dist(\Cluster',\Leaf),
\]
showing that $\Cluster'$ is $(\Leaf,\eta)$-admissible, which is impossible
because $\Cluster\in\MinC(\Leaf)$.  Thus,
\[
\#\MinC_\ell(\Leaf)\times\lambda TQ^{-\ell}
	\le\sum_{\Cluster\in\MinC_\ell(\Leaf)}\Len(\Cluster)
	<(1+\eta^{-1})\Lambda TQ^{-\ell+1}-\eta^{-1}\lambda TQ^{-\ell}
\]
and $\#\MinC_\ell(\Leaf)<(1+\eta^{-1})(\Lambda/\lambda)Q-\eta^{-1}$.

Now let $\ell=G$ and suppose for a contradiction that 
$a\le c-(1+\eta^{-1})\Lambda TQ^{-G}$.  Since
\[
b\le a+\Len(\Cluster)\le a+\Lambda TQ^{-G}\le c-\eta^{-1}\Lambda TQ^{-G}
\]
it follows that
\[
\Len(\Cluster)=b-a\le\eta\Lambda TQ^{-G}\le\eta(c-b)
	=\eta\Dist(\Cluster,\Leaf),
\]
so $\Cluster$ is $(\Leaf,\eta)$-admissible, which is impossible because
$\Cluster$ is a leaf.  Thus,
\[
\#\MinC_G(\Leaf)\times\lambda TQ^{-G}\le\sum_{\Cluster\in\MinC_G(\Leaf)}
	\Len(\Cluster)\le(1+\eta^{-1})\Lambda TQ^{-G}
\]
and $\#\MinC_G(\Leaf)\le(1+\eta^{-1})(\Lambda/\lambda)$.
\end{proof}

As a straight forward consequence, we obtain the desired operation counts.

\begin{theorem}\label{thm: operations}
If $\Tree$ is $(G,Q)$-uniform, then the right-hand side 
of~\eqref{eq: fastsum} can be computed for~$1\le n\le N$ in
order $r\eta^{-1}MN(QG+NQ^{-G})$ operations.
\end{theorem}
\begin{proof}
If $\Gen(\Cluster)=\ell$, then the number of subintervals in~$\Cluster$
is $NQ^{-\ell}$ so computing $\Sigma_n(\Cluster)$ requires
$O(NQ^{-\ell}M)$~operations.  Since 
\[
\#\Near(\Leaf_n)=\#\MinC_G(\Leaf_n)=O(\eta^{-1})
\]
and $\ell=G$ whenever $\Cluster\in\Near(\Leaf_n)$, we see that the 
total cost for all of the near-field sums is $O(\eta^{-1}NQ^{-G}M)$.
The sum~$\Sigmatilde_n(\Cluster,V)$ costs $O(rM)$~operations, and since
\[
\#\Far(\Leaf_n)=\sum_{\ell=1}^{G-1}\#\MinC_\ell(\Leaf_n)=O(\eta^{-1}QG)
\]
the total cost for all of the far-field sums is 
$O(\eta^{-1}QGrM)$~operations.  In addition, 
computing $\Psi_p(\Cluster,V)$ for every~$\Cluster$ 
with~$\Gen(\Cluster)=\ell$ costs $O(NM)$~operations, so computing this sum for
$1\le p\le r$ and $0\le\ell\le G-1$ costs $O(rGNM)$~operations.  
Thus, the overall cost for $N$~time
steps is of order $N\times(\eta^{-1}Q^{-G}NM+\eta^{-1}rQGM)+rGNM$~operations.
\end{proof}

If we choose $G=P=\log_QN$ so that $N=Q^G$ and each leaf contains only a single
subinterval, then the cost is $O(r\eta^{-1}QMN\log N)$, as claimed in
the Introduction.  In practice, the overheads associated with the 
tree data structure mean that it may be more efficient to choose $G<P$.
%%%%%%%%%%%%%%%%%%%%%%%%%%%%%%%%%%%%%%%%%%%%%%%%%%%%%%%%%%%%%%%%%%%%%
\section{Memory management}
\label{sec: memory}

From~\eqref{eq: number of nodes} we have
$\#(\Tree\setminus\Leaves)\le Q^G$, which implies that to store
$\Psi_p(\Cluster,\Utilde)$ for all $\Cluster\in\Tree\setminus\Leaves$~and 
$1\le p\le r$ we require $O(rQ^GM)$~memory locations.  Storing $\Utilde^n$
for~$1\le n\le N$ requires a further $O(NM)$~locations.  However, at the $n$th
time step only a small fraction of this memory is active, in the sense that
the data it holds play a role in computing~$\Utilde^n$.  
Figure~\ref{fig: clusters} illustrates a $(Q,G)$-uniform cluster tree with
$Q=2$~and $G=6$.  The black cluster is the current leaf~$\Leaf_n$, and the
red clusters belong to the minimal $(\Leaf_n,\eta)$-admissible 
cover~$\MinC(\Leaf_n)$.  As we will now explain, memory associated with 
the green and red clusters is active, whereas that associated with the 
blue and magenta clusters is not.  

\begin{figure}
\begin{center}
\includegraphics[scale=0.75, clip=true, trim=0mm 15mm 0mm 15mm]{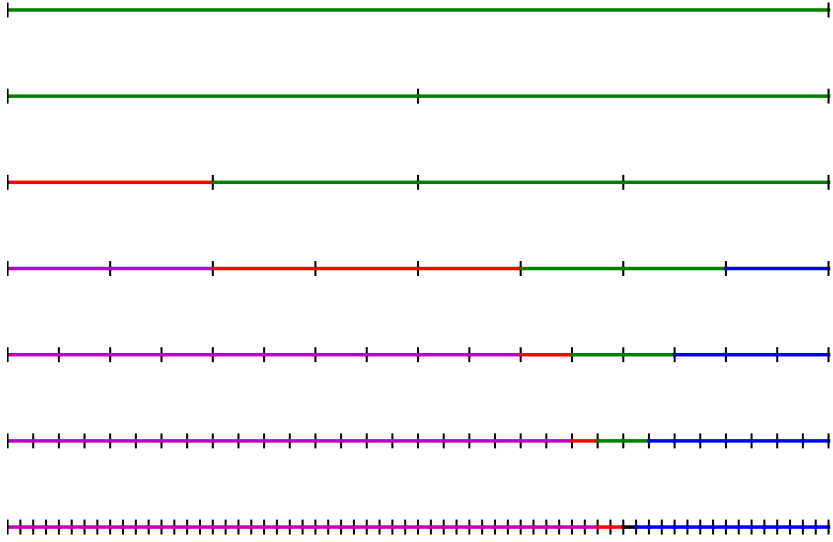}
\end{center}
\caption{Cluster tree with $\Leaf_n$ shown in black, the minimal
admissible cover~$\MinC(\Leaf_n)$ in red and the other active clusters
in green.  The blue clusters are not yet active and the magenta clusters
are no longer active.}
\label{fig: clusters}
\end{figure}

For each cluster~$\Cluster\in\Tree$, either there is no~$n$
such that $\Cluster\in\MinC(\Leaf_n)$, or else there is
a unique smallest~$n=\nmin(\Cluster)$ such that $\Cluster\in\MinC(\Leaf_n)$.
Moreover, if $\Cluster\notin\MinC(\Leaf_n)$ for some $n>\nmin(\Cluster)$,
then an ancestor of~$\Cluster$ must belong to~$\MinC(\Leaf_n)$
and we have $\Cluster\notin\MinC(\Leaf_{n'})$ for every $n'>n$.
Hence, there is also a unique $n=\nmax(\Cluster)$ such that 
\[
\Cluster\in\MinC(\Leaf_n)
\quad\text{if and only if}\quad
\nmin(\Cluster)\le n\le\nmax(\Cluster),
\]
so, if $\Cluster$ is not a leaf, the sums
\[
\Psi(\Cluster,\Utilde)=\Psi(\Cluster)
	=[\Psi_1(\Cluster),\Psi_2(\Cluster),\ldots, \Psi_r(\Cluster)]
\]
contribute to the far-field sum~$\Sigmatilde_n(\Cluster,\Utilde)$ if
and only if $\nmin(\Cluster)\le n\le\nmax(\Cluster)$.  We can therefore
deallocate the $O(rM)$~memory locations used to store~$\Psi(\Cluster)$
once $\Utilde^n$ has been computed for~$n=\nmax(\Cluster)$.

For each~$n$, define a subtree 
\[
\Tree_n=\{\,\Cluster\in\Tree\setminus\Leaves:
	\text{$\bigcup\Cluster$ intersects $I_n$}\,\},
\]
so that $\Psi_p(\Cluster,\Utilde)$ includes a term in~$\Utilde^n$
if and only if $\Cluster\in\Tree_n$.  In Algorithm~\ref{alg: time stepping},
after computing $\Utilde^n$ we update all far-field sums~$\Psi_p(\Cluster)$ 
with~$\Cluster\in\Tree_n$, so that $\Utilde^n$ is subsequently needed only 
for computing near-field sums.  In this way, we can deallocate the
$O(NQ^{-G}M)$~memory locations used to store $U^j$ for $I_j\in\Leaf$
once $\Utilde^n$ has been computed for~$n=\nmax(\Leaf)$.
Algorithm~\ref{alg: free} defines a recursive procedure $\free(\Cluster,V)$ 
that deallocates the memory associated with the children of~$\Cluster$, 
and with their descendants if not already deallocated.

\begin{algorithm}
\caption{Time stepping and memory management.}
\label{alg: time stepping}
\begin{algorithmic}
\FOR{$n=1$ to $N$}
  \STATE Find $\MinC(\Leaf_n)$ using Algorithm~\ref{alg: divide}
  \FORALL{$\Cluster\in\MinC(\Leaf_n)\setminus\Leaves$}
    \FORALL{$\Cluster'$ a child of $\Cluster$}
      \STATE $\free(\Cluster',\Utilde)$
    \ENDFOR
  \ENDFOR
  \STATE Compute $\Btilde^n\Utilde$ using \eqref{eq: fastsum} 
  \STATE Allocate $\Utilde^n$ and solve \eqref{eq: perturbed} 
  \STATE Write $\Utilde^n$ to disk
  \FORALL{$\Cluster\in\Tree_n$}
    \IF{$\Psi(\Cluster)$ is not allocated}
      \STATE Allocate $\Psi(\Cluster)$ and initialize to 0.
    \ENDIF
    \FORALL{$p\in\{1,2,\ldots,r\}$}
      \STATE $\Psi_p(\Cluster)=\Psi_p(\Cluster)+\psi_{pn}\Utilde^n$
    \ENDFOR
  \ENDFOR
\ENDFOR
\end{algorithmic}
\end{algorithm}

\begin{algorithm}
\caption{$\free(\Cluster, V)$}
\label{alg: free}
\begin{algorithmic}
\IF{$\Cluster\in\Leaves$}
  \FORALL{$I_j\in\Cluster$}
    \STATE{Deallocate $V^j$}
  \ENDFOR
\ELSIF{$\Psi(\Cluster)$ is allocated}
  \FORALL{$\Cluster'$ a child of $\Cluster$}
    \STATE $\free(\Cluster',V)$
  \ENDFOR
  \STATE Deallocate $\Psi(\Cluster)$
\ENDIF
\end{algorithmic}
\end{algorithm}

\begin{theorem}\label{thm: active memory}
The number of active memory locations used during the execution of
Algorithm~\ref{alg: time stepping} is never more than $O(r\eta^{-1}QGM)$.
\end{theorem}
\begin{proof}
Suppose that the memory associated with~$\Cluster$ is active
during the $n$th time step, and that $\bigcup\Cluster=(a,b]$
and $\bigcup\Leaf_n=(c,d]$.  (So in Figure~\ref{fig: clusters}, $\Cluster$
is green or red or black.)  If $\Gen(\Cluster)=\ell$, then by 
Lemma~\ref{lem: MinC count},
\[
c-(1+\eta^{-1})\Lambda TQ^{-\ell+1}<a<b\le d+\Lambda TQ^{-\ell},
\]
and since $d-c=\Len(\Leaf_n)\le\Lambda TQ^{-G}$~and 
$\Len(\Cluster)\ge\lambda TQ^-\ell$, we see that the number of such 
clusters is at most
\begin{align*}
\frac{(d-c)+\Lambda TQ^{-\ell}[1+(1+\eta^{-1})Q]}{\lambda TQ^{-\ell}}
	&\le\frac{\Lambda}{\lambda}\Bigl(Q^{\ell-G}+1+(1+\eta^{-1})Q\Bigr)\\
	&\le \frac{\Lambda}{\lambda}(1+\eta^{-1})(Q+1).
\end{align*}
Storing $\Psi_p(\Cluster,\Utilde)$ requires $rM$~memory locations, so
the desired estimate follows after adding the contributions 
for~$1\le\ell\le G$.
\end{proof}

Since $G\le P=\log_QN$, Theorem~\ref{thm: active memory} justifies the
claim in the Introduction that the memory requirements are proportional
to~$M\log N$.  Theorems \ref{thm: operations}~and 
\ref{thm: active memory} show that --- for a given choice of $M$~and $N$ and 
a given cluster tree --- the computational cost, both with respect 
to the number of operations and to the number of active memory locations, 
is proportional to~$r/\eta$.  At the same time, by Theorem~\ref{thm: error}, 
to achieve the desired accuracy we must ensure that $(r+1)(\eta/2)^r$ 
is sufficiently small.  The next result shows the relation between 
$r$~and $\eta$ that is optimal in the sense of achieving a given accuracy 
for the least computational cost.

\begin{proposition}\label{prop: eta}
For a given~$\delta>0$, the ratio~$r/\eta$ is minimised subject
to the constraint $(r+1)(\eta/2)^r=\delta$ by choosing
\begin{equation}\label{eq: optimal eta}
\eta=2\exp\biggl(-\frac{r+2}{r+1}\biggr)	
	=2\biggl(\frac{\delta}{r+1}\biggr)^{1/r}.
\end{equation}
\end{proposition}
\begin{proof}
Introducing the Lagrangian $L=r\eta^{-1}+\mu(r+1)(\eta/2)^r$, we obtain the
necessary conditions
\[
\eta^{-1}+\mu(\eta/2)^r[1+(1+r)\log(\eta/2)]=0
\quad\text{and}\quad
-r\eta^{-2}+\mu(r+1)r\eta^{r-1}2^{-r}=0,
\]
so $\mu(\eta/2)^r=\eta^{-1}/(r+1)$ and $r+2+(r+1)\log(\eta/2)=0$.
\end{proof}

Thus, we should choose successive values of~$r$ until the second 
inequality in~\eqref{eq: consistent error} holds, with~$\eta$ given
by~\eqref{eq: optimal eta}.  Since $\eta^{-1}\le e^2/2$, 
the computational cost is then proportional to~$r$.  
%%%%%%%%%%%%%%%%%%%%%%%%%%%%%%%%%%%%%%%%%%%%%%%%%%%%%%%%%%%%%%%%%%%%%
\section{Numerical example}
\label{sec: example}

\begin{table}
\begin{center}
\newcommand{\OS}{\phantom{0}}
\begin{tabular}{c|c|ccc}
&Slow&\multicolumn{3}{c}{Fast}\\
\hline
       $r$&    ---     &           4&           5&           6\\
    $\eta$&    ---     &      0.6024&      0.6228&      0.6378\\
     Error&   0.129E-03&   0.136E-02&   0.129E-03&   0.129E-03\\[2\jot]
%        E1&   0.331E-03&   0.129E-03&   0.129E-03\\
%        E2&    0.136E-02&   0.845E-04&   0.118E-03\\
     Setup&  \OS49.6\,s&  \OS0.57\,s&  \OS0.57\,s&     \OS0.62\,s\\
       RHS&    910.9\,s&    15.45\,s&    17.68\,s&       20.55\,s\\
    Solver&\OS\OS7.2\,s&  \OS6.96\,s&  \OS6.84\,s&     \OS6.68\,s\\
     Total&    967.7\,s&    22.97\,s&    25.10\,s&       27.85\,s \\
\hline
\end{tabular}
\end{center}
\caption{Performance of slow and fast methods with $N=16000$ time steps
and $M=6241$ spatial degrees of freedom.}
\label{tab: results}
\end{table}

Consider a simple test problem in $d=2$~spatial dimensions, 
with~$\nu=1/2$, $T=6$, $\Omega=(0,1)\times(0,1)$ and 
homogeneous Dirichlet boundary conditions~\eqref{eq: Dirichlet}.  
We take $K=1/(2\pi^2)$ so that the smallest eigenvalue of the elliptic
operator~$Au=-\nabla\cdot(K\nabla u)$ is $\lambda_{11}=1$.
We choose the initial data and source term 
\[
u_0=\varphi_{11}
\quad\text{and}\quad
f(t)=(1+\sin\pi t)\varphi_{11},
\]
where $\varphi_{11}(x)=(\sin\pi x_1)(\sin\pi x_2)$ is an eigenfunction 
of~$A$ with eigenvalue~$\lambda_{11}$.  The exact solution 
of~\eqref{eq: integrode} then has the separable 
form~$u(x,t)=u_{11}(t)\varphi_{11}(x)$, and we can
compute the time-dependent factor~$u_{11}(t)$ to high accuracy by
applying Gauss quadrature to an integral 
representation~\cite[Section~6]{McLeanThomee2010}.  Moreover,
the regularity estimates~\eqref{eq: regularity} hold with $\sigma=1$,
so by~\eqref{eq: DG error} the $L_2$-error in~$U^n$ is of 
order~$(k+h^2)\lg(t_n/t_1)$.

Table~\ref{tab: results} shows some results of computations performed using
a single-threaded Fortran code running on a desktop PC with an Intel 
Core-i7~860 processor (2.80GHz) and 8GB of RAM.  In all cases the spatial 
discretization used bilinear finite elements on a uniform 
$80\times80$~rectangular mesh, 
so the number of degrees of freedom was~$M=79^2=6241$.  
We solved the linear system~\eqref{eq: linear system} using fast
sine transforms.  Taking $N=16000$~time steps, we first computed the (slow) 
DG solution~$U$ and then the perturbed (fast) solution~$\Utilde$ 
for $r=4$, 5, 6 choosing $\eta$ as in
Proposition~\ref{prop: eta}.  The table shows the maximum nodal error
$\max_{1\le n\le N, 1\le m\le M}|U^n_m-u(x_m,t_n)|$, and also the CPU times
in seconds, broken down into three parts.  The setup phase covers
computing the $\beta_{nj}$ or $\betatilde_{nj}$, and for the fast method
the cost of constructing the cluster tree and admissible covers.  The
RHS phase covers the computation of the right-hand side 
of~\eqref{eq: linear system}, and the solver phase is the total CPU time
used by the elliptic solver.

The cluster tree was $(Q,G)$-uniform for
$Q=2$~and $G=10$, so there were $2^G=1024$ leaves.  We see from the table
that if $r=5$ then the fast summation algorithm evaluates the right-hand 
side (RHS) in 17.7~seconds, compared to 911~seconds for a direct evaluation,
while maintaining the accuracy of the DG solution.
%%%%%%%%%%%%%%%%%%%%%%%%%%%%%%%%%%%%%%%%%%%%%%%%%%%%%%%%%%%%%%%%%%%%%
\appendix
\section{A lower bound}
\label{sec: lower bound}
For any real-valued, piecewise-constant~$V$, 
\[
\sum_{n=1}^Nk_n(\Bbar^n V)V^n=\int_0^T\B V(t)V(t)\,dt
\quad\text{and}\quad
\sum_{n=1}^Nk_n(V^n)^2=\int_0^TV(t)^2\,dt,
\]
so the next theorem shows that \eqref{eq: lower} holds.

\begin{theorem}
If $v:[0,T]\to\R$ is piecewise $C^1$ then
\[
\int_0^T\B v(t) v(t)\,dt\ge\rho_\nu T^{\nu-1}\int_0^Tv(t)^2\,dt,
\]
where
\begin{equation}\label{eq: rho_nu}
\rho_\nu=\pi^{1-\nu}\,\frac{(1-\nu)^{1-\nu}}{(2-\nu)^{2-\nu}}\,
	\sin(\tfrac12\pi\nu).
\end{equation}
\end{theorem}
\begin{proof}
The assumption that $v$ is piecewise $C^1$ ensures $\B v$ is
continuous except for weak singularities at the breakpoints of~$v$.
Using the substitution~$t=\tau T$ for~$0<\tau<1$, we see that it 
suffices to deal with the case~$T=1$.  
Denote the Laplace transform of~$u$ by
\[
\hat u(z)=\int_0^\infty e^{-zt}u(t)\,dt\quad\text{for $\Re z\ge0$,}
\]
and observe that, because $\hat\omega_\nu(z)=z^{-\nu}$, 
if we extend $v$ by zero outside the interval~$[0,1]$, then
\[
\widehat{\B v}(z)=z\,\hat\omega_\nu(z)\hat v(z)=z^{1-\nu}\hat v(z).
\]
Applying the Plancherel Theorem, and noting that 
$\overline{\hat v(z)}=\hat v(\bar z)$ because $v$ is real-valued, we have
\[
\int_0^\infty u(t)v(t)\,dt=\frac{1}{2\pi}\int_{-\infty}^\infty
	\hat u(iy)\hat v(-iy)\,dy.
\]
In particular, 
\begin{equation}\label{eq: v norm}
\int_0^1 v(t)^2\,dt=\frac{1}{\pi}\int_0^\infty|\hat v(iy)|^2\,dy
\end{equation}
and
\begin{equation}\label{eq: Bv v}
\begin{aligned}
\int_0^1\B v(t)v(t)\,dt&=\int_0^\infty\B v(t)v(t)\,dt
	=\frac{1}{2\pi}\int_{-\infty}^\infty(iy)^{1-\nu}
		\hat v(iy)\hat v(-iy)\,dy\\
	&=\frac{\sin\tfrac12\pi\nu}{\pi}
	\int_0^\infty y^{1-\nu}|\hat v(iy)|^2\,dy.
\end{aligned}
\end{equation}

The estimate
\[
|\hat v(iy)|^2\le\biggl|\int_0^1 e^{-iyt}v(t)\,dt\biggr|^2
	\le\int_0^1 v(t)^2\,dt
\]
implies that, for any~$\epsilon>0$,
\[
\int_0^\epsilon |\hat v(iy)|^2\,dy\le\epsilon\int_0^1v(t)^2\,dt,
\]
and therefore by~\eqref{eq: v norm},
\[
\int_0^1 v(t)^2\,dt\le\frac{\epsilon}{\pi}\int_0^1 v(t)^2\,dt
	+\frac{1}{\pi}\int_\epsilon^\infty|\hat v(iy)|^2\,dy.
\]
Thus, for~$0<\epsilon<\pi$, 
\[
\biggl(1-\frac{\epsilon}{\pi}\biggr)\int_0^1 v(t)^2\,dy
	\le\frac{1}{\pi}\int_\epsilon^\infty
		(y/\epsilon)^{1-\nu}|\hat v(iy)|^2\,dy
	\le\frac{\epsilon^{\nu-1}}{\pi}\int_0^\infty 
		y^{1-\nu}|\hat v(iy)|^2\,dy,
\]
which, in combination with~\eqref{eq: Bv v}, implies that the desired
inequality holds with
\[
\rho_\nu=\epsilon^{1-\nu}(\pi-\epsilon)\,
	\frac{\sin\tfrac12\pi\nu}{\pi}.
\]
The choice~$\epsilon=\pi(1-\nu)/(2-\nu)$ maximises $\rho_\nu$ and gives
the formula~\eqref{eq: rho_nu}.
\end{proof}

%%%%%%%%%%%%%%%%%%%%%%%%%%%%%%%%%%%%%%%%%%%%%%%%%%%%%%%%%%%%%%%%%%%%%

\ifpreprint
\section{Computing the weights}
\label{sec: eval}
Since the diagonal weights present no difficulty, we assume throughout
this appendix that $1\le j\le n-1$.  Denoting the distance between the 
centres of $I_j$~and $I_n$ by $\Delta_{nj}=t_{n-1/2}-t_{j-1/2}$, we see 
from~\eqref{eq: beta_nj repn} that
\[
\beta_{nj}=-\int_{-k_j/2}^{k_j/2}\int_{-k_n/2}^{k_n/2}
	\omega_{\nu-1}(\Delta_{nj}+t+s)\,dt\,ds
\]
and so
\begin{equation}\label{eq: beta_nj B}
\beta_{nj}=-\int_{-k_j/2}^{k_j/2}B_{\nu-1}(\Delta_{nj}+s,k_n)\,ds
=-\int_{-k_n/2}^{k_n/2}B_{\nu-1}(\Delta_{nj}+t,k_j)\,dt.
\end{equation}
Although we can easily evaluate these integrals analytically, the
resulting expressions are susceptible to loss of precision when 
$k_j$~and $k_n$ are small compared to~$\Delta_{nj}$.

Consider the problem of computing
\[
B_\mu(t,k)=\omega_{1+\mu}(t+\tfrac12 k)-\omega_{1+\mu}(t-\tfrac12 k)
\]
when $k$ is small compared to~$t$, so that we have a difference of
nearly equal numbers.  The C99 standard library provides the functions
$\expmone(x)$~and $\logonep(x)$ that approximate
$e^x-1$~and $\log(1+x)$ accurately even when~$x$ is close to zero, so
we can avoid the loss of precision by noting that
\begin{equation}\label{eq: B_mu D_mu}
B_\mu(t,k)=\omega_{1+\mu}(t+\tfrac12k)
	D_\mu\biggl(\frac{k}{t+\tfrac12 k}
	\biggr)\quad\text{where $D_\mu(x)=1-(1-x)^\mu$,} 
\end{equation}
and evaluating $D_\mu(x)$ as $-\expmone\bigl(\mu\,\logonep(-x)\bigr)$.
  
However, even though we can compute $B_\mu(t)$ accurately, we still
face the problem that
\begin{equation}\label{eq: beta_nj direct}
\beta_{nj}=B_\nu(\Delta_{nj}-\tfrac12k_j,k_n)
	-B_\nu(\Delta_{nj}+\tfrac12k_j,k_n)
\end{equation}
is again a difference of nearly equal numbers, as is the alternative
formula
\[
\beta_{nj}=B_\nu(\Delta_{nj}-\tfrac12k_n,k_j)
	-B_\nu(\Delta_{nj}+\tfrac12k_n,k_j),
\]
or equivalently
\begin{align*}
\beta_{nj}&=\omega_{1+\nu}(t_n-t_j)D_\nu\biggl(
		\frac{k_n}{t_n-t_j}\biggr)
	 -\omega_{1+\nu}(t_n-t_{j-1})D_\nu\biggl(
	 	\frac{k_n}{t_n-t_{j-1}}\biggr)\\
	&=\omega_{1+\nu}(t_{n-1}-t_{j-1})D_\nu\biggl(
		\frac{k_j}{t_{n-1}-t_{j-1}}\biggr)
	 -\omega_{1+\nu}(t_n-t_{j-1})D_\nu\biggl(
	 	\frac{k_j}{t_n-t_{j-1}}\biggr).
\end{align*}
When~$k_j$ is small compared to~$\Delta_{nj}$, the following sum
gives a more accurate value for the weight~$\beta_{nj}$.

\begin{theorem}
If $1\le j\le n-1$, then there exists $t_{nj}^*\in I_j$
such that
\[
\beta_{nj}=-\sum_{p=0}^{r-1}
	\frac{B_{\nu-2p-1}(\Delta_{nj},k_n)}{(2p+1)!2^{2p}}\,k_j^{2p+1}
-\frac{B_{\nu-2r-1,j}(t_{n-1/2}-t_{nj}^*,k_n)}{(2r+1)!2^{2r}}\,k_j^{2r+1}.
\]
\end{theorem}
\begin{proof}
We use the first integral representation in~\eqref{eq: beta_nj B}.
The Taylor expansion
\begin{multline*}
B_{\nu-1}(\Delta_{nj}+s,k_n)
	=\sum_{p=0}^{2r-1}B_{\nu-1-p}(\Delta_{nj},k_n)\,\frac{s^p}{p!}\\
	+\int_0^s\frac{(s-t)^{2r-1}}{(2r-1)!}\,
		B_{\nu-1-2r}(\Delta_{nj}+t,k_n)\,dt
\end{multline*}
implies that
\[
\int_{-\tfrac12 k_j}^{\tfrac12 k_j}B_{\nu-1}(\Delta_{nj}+s,k_n)\,ds
	=\sum_{p=0}^{r-1}B_{\nu-1-2p}(\Delta_{nj},k_n)
		\frac{2(\tfrac12 k_j)^{2p+1}}{(2p+1)!}+E_r,
\]
with the error term given by~$E_r=E_r^++E_r^-$, where
\[
E_r^\pm=\int_0^{\tfrac12 k_j}\int_0^s\frac{(s-t)^{2r-1}}{(2r-1)!}
	B_{\nu-1-2r}(\Delta_{nj}\pm t,k_n)\,dt\,ds.
\]
By the Integral Mean Value Theorem, there exists 
$\theta_{nj}^\pm\in[0,\tfrac12 k_j]$ such that
\begin{multline*}
%\begin{align*}
E_r^\pm=B_{\nu-1-2r}(\Delta_{nj}\pm\theta_{nj}^\pm,k_n)\int_0^{\tfrac12 k_j}
	\int_0^s \frac{(s-t)^{2r-1}}{(2r-1)!}\,dt\,ds\\
	=B_{\nu-1-2r}(\Delta_{nj}\pm\theta_{nj}^\pm,k_n)\,
		\frac{(\tfrac12 k_j)^{2r+1}}{(2r+1)!}.
%\end{align*}
\end{multline*}
Since $\Delta_{nj}\pm\theta_{nj}^\pm=t_{n-1/2}-(t_{j-1/2}\mp\theta_{nj}^\pm)$
and $t_{j-1/2}\mp\theta_{nj}^\pm\in I_j$, by the Intermediate Value Theorem 
there exists $t_{nj}^*\in I_j$ such that
\[
B_{\nu-1-2r}(\Delta_{nj}+\theta_{nj}^+,k_n)
	+B_{\nu-1-2r}(\Delta_{nj}-\theta_{nj}^-,k_n)
	=2B_{\nu-1-2r}(t_{n-1/2}-t_{nj}^*,k_n).
\]
\end{proof}

By starting from the \emph{second} integral representation 
in~\eqref{eq: beta_nj B}, we obtain an expansion in odd powers
of~$k_n$, instead of~$k_j$.  For practical meshes, we generally have
$k_j\le k_n$, so the series in the theorem will be preferable.

To determine the speed of convergence of the series, denote the $p$th
term by
\[
b_p=-\frac{B_{\nu-2p-1}(\Delta_{nj},k_n)}{(2p+1)!2^{2p}}\,k_j^{2p+1},
\]
and note that since $D_{-\mu}(x)=-D_\mu(x)/(1-x)^\mu$,
\[
-B_{\nu-2p-1}(\Delta_{nj},k_n)=\omega_{\nu-2p}(\Delta_{nj}+\tfrac12 k_n)
	\frac{D_{2p+1-\nu}(x)}{(1-x)^{2p+1-\nu}}
	\quad\text{for $x=\frac{k_n}{\Delta_{nj}+\tfrac12 k_n}$.}
\]
We find that the ratio of successive terms is
\[
\frac{b_{p+1}}{b_p}=\frac{1}{4}\,\frac{2p+2-\nu}{2p+3}\,
	\frac{2p+1-\nu}{2p+2}\,\frac{D_{2p+3-\nu}(x)}{D_{2p+1-\nu}(x)}
	\biggl(\frac{k_j}{\Delta_{nj}-\tfrac12k_n}\biggr)^2,
\]
and, since $D_\mu(x)\to1$ as $\mu\to\infty$ for $0<x<1$, 
\[
\lim_{p\to\infty}\frac{b_{p+1}}{b_p}
	=\biggl(\frac{k_j}{2\Delta_{nj}-k_n}\biggr)^2.
\]
For instance, in the case of a uniform grid $t_n=nk$, this limiting 
ratio is $[2(n-j)-1]^{-2}$, giving acceptable convergence for $j\le n-2$.
If $j=n-1$, then the limiting ratio is $k_{n-1}/(k_{n-1}+k_n)$,
so the convergence is relatively slow.  We see 
from~\eqref{eq: beta_nj direct} that
\begin{align*}
\beta_{n,n-1}&=\omega_{1+\nu}(k_n)+\omega_{1+\nu}(k_{n-1})
	-\omega_{1+\nu}(k_n+k_{n-1})\\
	&=\omega_{1+\nu}(k_n)\bigl[1+x^\nu-(1+x)^\nu\bigr],
	\quad\text{for $x=k_{n-1}/k_n$,}
\end{align*}
and from symmetry we may assume $k_{n-1}\le k_n$ and thus $0<x\le1$.
To evaluate the difference in square brackets, we write
\[
(1+x)^\nu=1+y^\nu\quad\text{where}\quad y=\exp\biggl(
	\frac{\log[(1+x)^\nu-1]}{\nu}\biggr)
\]
computing $(1+x)^\nu-1$ as $\expmone\bigl(\nu\,\logonep(x)\bigr)$.
In this way,
\[
1+x^\nu-(1+x)^\nu=x^\nu-y^\nu=y^\nu\bigl[(x/y)^\nu-1\bigr],
\]
and we compute $(x/y)^\nu-1$ as $\expmone[\nu\log(x/y)]$.

We remark that in the special case~$\nu=1/2$ one can evaluate $\beta_{nj}$
more easily.  Firstly,
\[
B_{1/2}(t,k)=\frac{1}{\Gamma(3/2)}\,\frac{k}%
{\sqrt{t+\tfrac12k}+\sqrt{t-\tfrac12k}},
\]
and if we write 
$R_{*\diamond}=\sqrt{\Delta_{nj}*\tfrac12k_j\diamond\tfrac12k_n}$
for $*$, $\diamond\in\{+,-\}$, then
\begin{align*}
\beta_{nj}&=B_{1/2}(\Delta_{nj}-\tfrac12k_j,k_n)
	-B_{1/2}(\Delta_{nj}+\tfrac12k_j,k_n)\\
	&=\frac{k_n}{\Gamma(3/2)}\,\biggl(
	\frac{1}{R_{-+}+R_{--}}-\frac{1}{R_{++}+R_{+-}}\biggr)\\
	&=\frac{k_n}{\Gamma(3/2)}\,
\frac{(R_{++}-R_{-+})+(R_{+-}-R_{--})}{(R_{-+}+R_{--})(R_{++}+R_{+-})}\\
	&=\frac{k_nk_j}{\Gamma(3/2)}\,\frac{1}{
	(R_{-+}+R_{--})(R_{++}+R_{+-})}\biggl(
	\frac{1}{R_{++}+R_{-+}}+\frac{1}{R_{+-}+R_{--}}\biggr).
\end{align*}
%%%%%%%%%%%%%%%%%%%%%%%%%%%%%%%%%%%%%%%%%%%%%%%%%%%%%%%%%%%%%%%%%%%%%
\section{Computing the Taylor approximations}
\label{sec: computing low}
Recall from Theorem~\ref{thm: Taylor expansion} that
\[
\phi_{pn}=(-1)^pB_{\nu-p}(t_{n-1/2}-\smid,k_n)
\]
so \eqref{eq: B_mu D_mu} gives
\[
\phi_{pn}=(-1)^p\omega_{\nu-p+1}(t_n-\smid)D_{\nu-p}(x)
	\quad\text{where $x=\dfrac{k_n}{t_n-\smid}$.}
\]
Since $D_{\nu-p}(x)=-(1-x)^{\nu-p}D_{p-\nu}(x)$
and $1-x=(t_{n-1}-\smid)/(t_n-\smid)$, we have
\[
\phi_{pn}=\kappa_{pn}D_{p-\nu}(x),\quad\text{where}\quad
	\kappa_{pn}=(-1)^{p+1}
	\frac{(t_{n-1}-\smid)^{\nu-p}}{\Gamma(\nu-p+1)},
\]
and the $\kappa_{pn}$ can be computed via the recursion
\[
\kappa_{1n}=\frac{(t_{n-1}-\smid)^{\nu-1}}{\Gamma(\nu)}
\quad\text{and}\quad
\kappa_{p+1,n}=\frac{p-\nu}{t_{n-1}-\smid}\,\kappa_{p,n}
	\quad\text{for $1\le p\le r-1$.}
\]
Likewise,
\begin{align*}
\psi_{pj}&=B_p(t_{j-1/2}-\smid,k_j)=\frac{1}{p!}\bigl[
	(t_{j-1/2}-\smid+\tfrac12 k_j)^p-
	(t_{j-1/2}-\smid-\tfrac12 k_j)^p\bigr]\\
	&=\frac{k_j}{p!}\sum_{q=0}^{p-1} (t_j-\smid)^q
	(t_{j-1}-\smid)^{p-1-q},
\end{align*}
giving $\psi_{1n}=k_j$ and
\[
\psi_{p+1,j}=\frac{1}{p+1}\biggl((t_{j-1}-\smid)\psi_{pj}
	+\frac{k_j}{p!}\,(t_j-\smid)^p\biggr)
	\quad\text{for $1\le p\le r-1$.}
\]
%%%%%%%%%%%%%%%%%%%%%%%%%%%%%%%%%%%%%%%%%%%%%%%%%%%%%%%%%%%%%%%%%%%%%

%%%%%%%%%%%%%%%%%%%%%%%%%%%%%%%%%%%%%%%%%%%%%%%%%%%%%%%%%%%%%%%%%%%%%

\begin{thebibliography}{0}

\bibitem{CuestaLubichPalencia2006} 
Eduardo Cuesta, Christian Lubich and Cesar Palencia,
Convolution quadrature time discretization of fractional diffusion-wave
equations, 
\emph{Math. Comp.} 75: 673--696, 2006.

\bibitem{LopezFernandezLubichSchaedle2008} 
Mar\'ia L\'opez-Fernandez, Christian Lubich and Achim Sch\"adle,
Adaptive, fast, and oblivious convolution in evolution equations  with memory,
\emph{SIAM J. Sci. Comput.} 30: 1015--1037, 2008.

\bibitem{HackbuschNowak1989}
W.~Hackbusch and Z.~P.~Nowak,
On the fast matrix multiplication in the boundary element method
by panel clustering,
\emph{Numer. Math.} 54:463--491, 1989.

\bibitem{McLean2010}
William McLean,
Regularity of solutions to a time fractional diffusion equation,
\emph{ANZIAM J} 52: 123--138, 2010.

\bibitem{McLeanMustapha2009}
William McLean and Mustapha Kassem,
Convergence analysis of a discontinuous Galerkin method for a
sub-diffusion equation,
\emph{Numer. Algorithms} 52:~69--88, 2009.

\bibitem{McLeanThomee2010}
William McLean and Vidar Th\'omee,
Numerical solution via Laplace transforms of a fractional order
evolution equation, \emph{J. Integral Equations Appl.} 22: 57--94,
2010.

\bibitem{MetzlerKlafter2000}
Ralf Metzler and Joseph Klafter, The random walk's guide to anomalous
diffusion: a fractional dynamics approach,
\emph{Physics Reports} 339: 1--77, 2000.

\bibitem{SchaedleLopezFernandezLubich2006}
Achim Sch\"adle, Mar\'ia L\'opez-Fernandez and Christian Lubich,
Fast and oblivious convolution quadrature, 
\emph{SIAM J. Sci. Comput.} 28:421--438, 2006.
\end{thebibliography}
\end{document}